\documentclass[11pt,a4paper]{amsart}
\usepackage{enumitem}
\usepackage{amssymb}
\usepackage{hyperref}

\usepackage{tikz}
\usetikzlibrary{positioning,arrows}

\newcommand{\Aut}{\operatorname{Aut}}

\newcommand{\Z}{{\mathbb Z}}
\newcommand{\F}{{\mathbb F}}

\newcommand{\E}{{\mathcal E}}

\newcommand{\End}{\operatorname{End}}

\renewcommand{\Im}{\operatorname{Im}}

\newcommand{\id}{\operatorname{id}}

\newtheorem{theorem}{Theorem}[section]
\newtheorem{lemma}[theorem]{Lemma}
\newtheorem{corollary}[theorem]{Corollary}
\newtheorem{proposition}[theorem]{Proposition}

\newtheorem{question}[theorem]{Question}

\theoremstyle{definition}

	\theoremstyle{definition}
	\newtheorem{rem&notation}[theorem]{Remark and notation}

\newtheorem{remark}[theorem]{Remark}
\newtheorem{example}[theorem]{Example}

\theoremstyle{definition}
\newtheorem{definition}[theorem]{Definition}

\numberwithin{equation}{section}

\title[Stable homotopy classes of self-maps of polyhedra]{The ring of stable homotopy classes of self-maps of $A_n^2$-polyhedra}

\author{David M\'endez}
\address{School of Mathematical Sciences, University of Southampton, SO17 1BJ, United Kingdom}
\email{D.Mendez-Martinez@soton.ac.uk}
\thanks{This work was supported by The Alan Turing Institute under the EPSRC grant EP/N510129/1. The author was partially supported by Ministerio de Econom\'ia y Competitividad (Spain) grants MTM2016-79661-P and MTM2016-78647-P. The author is thankful to C.~Costoya and A.~Viruel for the fruitful discussions we had on the contents of this paper}

\makeatletter
\@namedef{subjclassname@2020}{%
  \textup{2020} Mathematics Subject Classification}
\makeatother

\AtBeginDocument{%
   \def\MR#1{}
}

\begin{document}


\begin{abstract}
	We raise the problem of realisability of rings as $\{X,X\}$ the ring of stable homotopy classes of self-maps of a space $X$. By focusing on $A_n^2$-polyhedra, we show that the direct sum of three endomorphism rings of abelian groups, one of which must be free, is realisable as $\{X,X\}$ modulo the acyclic maps. We also show that $\F_p^3$ is not realisable in the setting of finite type $A_n^2$-polyhedra, for $p$ any prime.
\end{abstract}

\maketitle


\section{Introduction}\label{sec:intro}

Let $X$ be a space and let $\E(X)$ denote the group of homotopy classes of self-homotopy equivalences of $X$. Many problems related to $\E(X)$ have been extensively studied, but among them, the question of realisability has received the most attention. This problem was proposed by Kahn in the 60s, and asks if every group appears as the group of self-homotopy equivalences of a topological space. It has been placed first to solve in \cite{Ark01} and has been looked into by many authors, \cite{Ark90,Kah76,Kah90,Rut97}. Although progress towards a solution to Kahn's problem was slow in the early years, it was recently proven that finite groups are realisable as groups of self-homotopy equivalences of rational spaces, \cite{CosVir14}.  

The group $\E(X)$ can also be regarded as the group of invertible elements of the monoid $[X,X]$ of homotopy classes of self-maps of $X$. Thus, Kahn's problem can be thought of as a particular case of the problem of realising monoids as $[X,X]$ for some space $X$. In \cite{CosMenVir19} we are able to use refinements to the constructions in \cite{CosMenVir18,CosVir14} to show that finite monoids with zero and with no non-trivial zero divisors are so realisable, \cite[Corollary 1.6]{CosMenVir19}.

In this paper we settle our interest in these realisability problems in the stable setting. Of course, the algebraic structures above can be defined for stable homotopy classes of maps. Furthermore, the monoid of stable self-maps of a space $X$, denoted $\{X,X\}$, admits a natural ring structure where the additive operation is the addition of maps coming from suspensions. Such ring is usually denoted by $\End(X)$. Consequently, the group of stable self-homotopy equivalences of $X$, denoted $\E^S(X)$, is just the group of units of $\End(X)$, and both of these algebraic structures associated to spectra have been studied by several authors, \cite{
Kah76,Joh72,Pav02,Pav10}. 

In our case, following the discussion above, we are interested in the following realisability question for the ring $\End(X)$.

\begin{question}\label{question}
	Which rings are realisable as $\End(X)$ or as some of its distinguished subgroups or quotients for $X$ a topological space?
\end{question}

This question has been extensively studied for the quotient of $\End(X)$ by its Jacobson radical, denoted $\operatorname{Rad}$, particularly in the case of finite $p$-torsion spectra. In such case, it is known that $\End(X)/\operatorname{Rad}$ is isomorphic to a countable product of matrix algebras over finite fields, \cite{Hub89}, and that if $X$ is atomic then such quotient is a finite field of characteristic $p$, \cite{AdaKuh89}. Furthermore, every finite product of finite algebras over finite fields of characteristic $p$ is so realisable, \cite{CraHubXu92}. 

These results are generally obtained by taking a careful look into the representation of $\End(X)$ in the homology groups of $X$, which hints at the importance of homology computations in a solution to Question \ref{question}. Following this lead, we turn to ideas of algebraic homotopy introduced by J.\ H.\ C.\ Whitehead and Baues. 

Namely, we know that homotopy types of certain polyhedra can be studied by purely algebraic tools in which the homology groups play a key role, \cite{Bau95}. In \cite{CosMenVir20} we used these ideas to study self-homotopy equivalences of $A_n^2$-polyhedra, $(n-1)$-connected $(n+2)$-dimensional CW-complexes, in the unstable setting. And it turns out that the results we used can be extended to the stable framework.

Indeed, the algebraic classification of these objects, introduced in \cite{Bau91,Whi50}, is provided by a functor $\Gamma$ from the homotopy category of $A_n^2$-polyhedra to a category whose objects are exact sequences of groups involving the homology groups of the polyhedra, see Definition \ref{definition:gammaSequencesn+2}. Then, if $X$ is an $A_n^2$-polyhedron, $\End\big(\Gamma(X)\big)$ can be endowed with a natural ring structure, for which we can prove the following.

\begin{proposition}\label{proposition:isoGammaSeq}
	Let $X$ be an $A_n^2$-polyhedron in the stable range. The map $H\colon\End(X)\to\End\big(\Gamma(X)\big)$ defined as $H(f) = \big(H_{n+2}(f),H_{n+1}(f), H_n(f)\big)\in\End\big(\Gamma(X)\big)$ (see Definition \ref{definition:homH}) is an epimorphism whose kernel consists in the homologically trivial self-maps of $X$.
\end{proposition}

Consequently, we can use the $\Gamma$ functor to algebraically study Question \ref{question} for the quotient of $\End(X)$ by the homologically trivial maps. For example, using Moore spaces we are able to show that endomorphism rings of abelian groups are realisable as $\End\big(\Gamma(X)\big)$ for $X$ an $A_n^2$-polyhedron, Proposition \ref{proposition:end1}. This contrasts with results on the entire ring $\End(X)$, particularly in the case of groups with $2$-torsion, see \cite[\textsection2.3]{Pav10}. We are later able to extend this result to the direct product of any two endomorphism rings of abelian groups, Proposition \ref{proposition:end2}. 

These two realisability results  for endomorphism rings are obtained by avoiding interaction between the homology groups of $X$ in the exact sequence $\Gamma(X)$. We then observe that we can make these groups interact to realise a pullback of certain subrings of $\End\big(H_n(X)\big)$ and $\End\big(H_{n+2}(X)\big)$, see Proposition \ref{proposition:pullback}. Finally, using ingredients in the results above we show our main realisability result:

\begin{theorem}\label{theorem:end3}
	Let $R$ be a ring such that $R\cong\End(G_1)\times\End(G_2)\times\End(G_3)$ for abelian groups $G_1$, $G_2$ and $G_3$ such that $G_1$ is free. Then, there is an $A_n^2$-polyhedron in the stable range $X$ such that $\End\big(\Gamma(X)\big)\cong R$.
\end{theorem}

Therefore, many interesting rings are realisable as $\End\big(\Gamma(X)\big)$ for $X$ an $A_n^2$-polyhedron. However, rings such as $\F_p^3$ for $p$ prime are not realisable using Theorem \ref{theorem:end3}, a fact that can easily be deduced from Remark \ref{remark:negativeRealisability}. And indeed, a careful examination of the algebraic objects associated to $A_n^2$-polyhedra leads us to conclude that this context may not be suitable to obtain a complete positive answer to Question \ref{question} for the considered quotient, at least in the setting of finite type $A_n^2$-polyhedra. Namely, we prove:
 
\begin{theorem}\label{theorem:main}
	Let $p$ be a prime. An $A_n^2$-polyhedron of finite type in the stable range such that $\End\big(\Gamma(X)\big)\cong\F_p^3$ does not exist.
\end{theorem}

\section{\texorpdfstring{$\Gamma$}{Gamma}-sequences and the ring of stable homotopy classes of self-maps}\label{sec:classification}

Recall that an $A_n^2$-polyhedron is an $(n+2)$-dimensional $(n-1)$-connected CW-complex. The homotopy types of these spaces have been classified by  Whitehead, \cite{Whi50}, and Baues, \cite{Bau91}. In this section we review how such classification can be used to study the ring of stable homotopy classes of self-maps of $A_n^2$-polyhedra.

Denote the homotopy category of $A_n^2$-polyhedra by $A_n^2$. Recall from \cite[10.2]{Bau95} that, as a consequence of Freudenthal's suspension theorem, there is a sequence of equivalences of categories
\[A_4^2\xrightarrow{\ \Sigma\ } A_5^2\xrightarrow{\ \Sigma\ }\cdots\xrightarrow{\ \Sigma\ } A_n^2\xrightarrow{\ \Sigma\ }A_{n+1}^2\xrightarrow{\ \Sigma\ }\cdots.\]
We can thus assume that $A_n^2$-polyhedra are in the stable range as long as $n\ge 4$. We will introduce the classification of homotopy types of these spaces in the range $n\ge 3$ and assume that $n\ge 4$ when we consider the stable range. We refer to \cite[\textsection2]{CosMenVir20} for a brief description of the case $n=2$ and to \cite[Ch. I, \textsection8]{Bau91} for the actual construction of the classification for every $n\ge 2$. 

Let $X$ be an $A_n^2$-polyhedron, $n\ge 3$. Consider the final part of the exact sequence of Whitehead associated to this polyhedron, \cite{Whi50}, which as a consequence of \cite[Theorem 2.1.22]{Bau96} looks as follows:
\begin{equation}\label{eq:gammaSeqX}
H_{n+2}(X)\xrightarrow{b_{n+2}} H_n(X)\otimes\Z/2\xrightarrow{\ i_n\ }\pi_{n+1}(X)\xrightarrow{h_{n+1}}H_{n+1}(X)\longrightarrow 0.
\end{equation}
Homotopy classes of $A_n^2$-polyhedra are classified by the isomorphism classes of these sort of sequences in the following category.

\begin{definition}{(\cite[Ch. IX, \textsection 4]{Bau89})}\label{definition:gammaSequencesn+2}
	Let $n\ge 3$ be an integer. The category $\Gamma$-sequences$^{n+2}$ is defined as follows. Its objects are triples of abelian groups $(H_{n+2}, H_{n+1}, H_n)$, $H_{n+2}$ free abelian, together with an exact sequence
	\[H_{n+2}\longrightarrow H_n\otimes\Z/2\longrightarrow\pi_{n+1}\longrightarrow H_{n+1}\longrightarrow 0.\]
	Its morphisms are triples of homomorphisms $f=(f_{n+2}, f_{n+1}, f_n)$, $f_i\colon H_i\to H_i'$, such that there is a group homomorphism $\Omega\colon\pi_{n+1}\to\pi_{n+1}'$ giving raise to a commutative diagram
	\begin{center}
		\begin{tikzpicture}
			\tikzset{node distance=0.15\textwidth, auto}
			\node(1) {$H_{n+2}$};
			\node[right of=1, xshift=0.9cm] (2) {$H_n\otimes\Z/2$};
			\node[right of=2, xshift=0.5cm] (3) {$\pi_{n+1}$};
			\node[right of=3, xshift=0.8cm] (4) {$H_{n+1}$};
			\node[right of=4] (5) {$0$};
			\node[below of=1, yshift=0.06cm] (6) {$H'_{n+2}$};
			\node[below of=2, yshift=0.06cm] (7) {$H'_n\otimes\Z/2$};
			\node[below of=3, yshift=0.06cm] (8) {$\pi'_{n+1}$};
			\node[below of=4, yshift=0.06cm] (9) {$H'_{n+1}$};
			\node[below of=5, yshift=0.06cm] (10) {$0.$};
			\draw[->,swap](1) to node {} (2);
			\draw[->](2) to node {} (3);
			\draw[->,swap](3) to node {} (4);
			\draw[->](4) to node {} (5);
			\draw[->](6) to node{} (7);
			\draw[->](7) to node {} (8);
			\draw[->](8) to node {} (9);
			\draw[->](9) to node {} (10);
			\draw[->](1) to node {$f_{n+2}$} (6);
			\draw[->] (2) to node {$f_n\otimes\Z/2$} (7);
			\draw[->] (3) to node {$\Omega$} (8);
			\draw[->] (4) to node {$f_{n+1}$} (9);
		\end{tikzpicture}
	\end{center}
	Objects and morphisms in this category are respectively called $\Gamma$-sequences and $\Gamma$-morphisms.
\end{definition}

Now, given $X$ an $A_n^2$-polyhedron, $n\ge 3$, its associated exact sequence \eqref{eq:gammaSeqX} is an object in $\Gamma$-sequences$^{n+2}$ which we call the $\Gamma$-sequence of $X$. We can then define a functor $\Gamma\colon A_n^2\to\Gamma$-sequences$^{n+2}$ as follows. To an $A_n^2$-polyhedron $X$ associate $\Gamma(X)$ the $\Gamma$-sequence of $X$, and to a continuous map $\alpha\colon X\to X'$ of $A_n^2$-polyhedra associate $\Gamma(\alpha) = \big(H_{n+2}(\alpha),H_{n+1}(\alpha),H_{n}(\alpha)\big)$, which is a $\Gamma$-morphism by defining $\Omega=\pi_{n+1}(\alpha)$. Baues proved the following result.

\begin{theorem}{\emph{(\cite[Ch. I, \textsection8]{Bau91})}}\label{theorem:classification}
	For every $n\ge 3$, the functor $\Gamma\colon A_n^2\to\Gamma$-sequences$^{n+2}$ is full, and every object in $\Gamma$-sequences$^{n+2}$ is the $\Gamma$-sequence of an $A_n^2$-polyhedron. In fact, there is a 1--1 correspondence between homotopy types of $A_n^2$-polyhedra and isomorphism classes of $\Gamma$-sequences.
\end{theorem}

In \cite{CosMenVir20}, we used this result to show that the automorphism group of the $\Gamma$-sequence of an $A_n^2$-polyhedron is isomorphic to a distinguished quotient of its group of self-homotopy equivalences. Using similar ideas, we show here that the $\Gamma$-sequence of an $A_n^2$-polyhedron can be used to study a distinguished quotient of its ring of stable homotopy classes of self-maps.

Let $X$ be a CW-complex in the stable range and denote $\End(X)=\{X,X\}$. Then $\End(X)$ can be endowed with a ring structure where the additive operation is the addition of maps coming from suspensions and the multiplicative operation is the composition of maps. In particular, for $f, g\in\End(X)$, $H_*(f+g) = H_*(f) + H_*(g)$. This hints that we may be able to extract information about $\End(X)$ from $\End\big(\Gamma(X)\big)$ if we were to provide it with the natural ring structure arising from the group operation of the homology groups of $X$. We do so now.

\begin{definition}\label{definition:homH}
	Let $\Gamma = \big(H_{n+2}\to H_n\otimes\Z/2\to\pi_{n+1}\to H_{n+1}\big)$ be an object of $\Gamma$-sequences$^{n+2}$. The ring of endomorphisms of $\Gamma$, $\End(\Gamma)$, is the ring whose elements are the endomorphisms of $\Gamma$ and where the additive and multiplicative operations are respectively the component-wise addition and composition of maps.

Now let $X$ be an $A_n^2$-polyhedron. There is a natural ring homomorphism
\begin{align*}
	H\colon\End(X)&\longrightarrow \End\big(\Gamma(X)\big)\\
	\alpha&\longmapsto \big(H_{n+2}(\alpha), H_{n+1}(\alpha), H_n(\alpha)\big)
\end{align*}
which is surjective as a consequence of Theorem \ref{theorem:classification}.
\end{definition}

By computing the kernel of the ring homomorphism $H$, we have the following result.

\begin{proposition}[Proposition \ref{proposition:isoGammaSeq}]\label{proposition:stableSelfEquiv}
	Let $X$ be an $A_n^2$-polyhedron in the stable range. There is a ring isomorphism $\End\big(\Gamma(X)\big)\cong\End(X)/\{\alpha\in\End(X)\mid H_*(\alpha)=0\}$.
\end{proposition}

\begin{remark}\label{remark:CMV2applicable}
	Note that the group of units of $\End\big(\Gamma(X)\big)$ is precisely the group of $\Gamma$-automorphisms of $\Gamma(X)$, which in \cite{CosMenVir20} we denoted by $\mathcal{B}^{n+2}(X)$. As a consequence of \cite[Proposition 2.5]{CosMenVir20}, this group is isomorphic to $\E(X)/\E_*(X)$. Therefore, the results obtained in \cite{CosMenVir20} are applicable to the group of units of $\End\big(\Gamma(X)\big)$, and can thus be immediately reformulated in terms of the stable group of self-homotopy equivalences. 
\end{remark}

\section{Realisability of rings as \texorpdfstring{$\End\big(\Gamma(X)\big)$}{End(Gamma(X))}}

In this section we prove our main realisability results. We begin by showing that we can use Moore spaces to realise any endomorphism ring of an abelian group as $\End\big(\Gamma(X)\big)$ for an $A_n^2$-polyhedron $X$. Note that Moore spaces also allowed us to realise automorphism groups as $\Aut\big(\Gamma(X)\big)$ in \cite[Example 3.3]{CosMenVir20}. 

\begin{proposition}\label{proposition:end1}
	Let $R$ be a ring such that $R\cong \End(G)$ for some abelian group $G$. Then, there is an $A_n^2$-polyhedron in the stable range such that $\End\big(\Gamma(X)\big)\cong R$. 
\end{proposition}

\begin{proof}
	Consider the Moore space $X=M(G, n+1)$, which is an $A_n^2$-polyhedron whose $\Gamma$-sequence is 
	\[H_{n+2}(X) = 0\longrightarrow \Gamma_n^1\big(H_n(X)\big) = 0 \longrightarrow G \xrightarrow{\ =\ } G\longrightarrow 0.\]
	As $H_{n+1}(X)$ is the only non-trivial homology group of $X$, $\End\big(\Gamma(X)\big)\le\End\big(H_{n+1}(X)\big)=\End(G)$. Furthermore, for $f\in\End(G)$, if we take $\Omega = f$, then $(\id, f,\id)\in\End\big(\Gamma(X)\big)$. Thus $\End\big(\Gamma(X)\big)\cong\End(G)\cong R$.
\end{proof}

\begin{remark}\label{remark:negativeRealisability}
	Although powerful, Proposition \ref{proposition:end1} is not enough to provide a positive answer to Question \ref{question}. For example, $\F_p^2$ is not the endomorphism ring of any abelian group, for $p$ any prime. Indeed, it is not hard to prove that any abelian group only having endomorphisms of (additive) order $p$ must be a $\F_p$-vector space. But one such vector space will only have a finite endomorphism ring if it is of finite dimension. The endomorphism ring of $\F_p^k$ is the ring of square matrices of order $k$ over $\F_p$, thus $\big|\End(\F_p^k)\big| = p^{k^2}$. However, $p^{k^2}\ne p^2$ for any $k$. 
\end{remark}

Now, by exploiting the other non-necessarily free homology group of an $A_n^2$-polyhedron, we easily prove the following result, which in particular allows us to realise $\F_p^2$.

\begin{proposition}\label{proposition:end2}
	Let $R$ be a ring such that $R\cong\End(G_1)\times\End(G_2)$ for abelian groups $G_1$ and $G_2$. Then, there is an $A_n^2$-polyhedron in the stable range $X$ such that $\End\big(\Gamma(X)\big)\cong R$.
\end{proposition}

\begin{proof}
	Let $X$ be the $A_n^2$-polyhedron whose $\Gamma$-sequence is the split short exact sequence
	\[0\longrightarrow G_1\otimes\Z/2\longrightarrow (G_1\otimes\Z/2)\times G_2\longrightarrow G_2\longrightarrow 0.\]	
	First, note that $\End\big(\Gamma(X)\big)\le \End(G_1)\times\End(G_2)$. On the other hand, given any $f_1\in\End(G_1)$ and $f_2\in\End(G_2)$ we have a commutative diagram
	\begin{center}
	\begin{tikzpicture}
		\tikzset{node distance=0.15\textwidth, auto}
		\node(1) {0};
		\node[right of=1, xshift=0.4cm] (2) {$G_1\otimes\Z/2$};
		\node[right of=2, xshift=1cm] (3) {$(G_1\otimes\Z/2)\times G_2$};
		\node[right of=3, xshift=0.8cm] (4) {$G_2$};
		\node[right of=4] (5) {$0$};
		\node[below of=1, yshift=0.06cm] (6) {$0$};
		\node[below of=2, yshift=0.06cm] (7) {$G_1\otimes\Z/2$};
		\node[below of=3, yshift=0.06cm] (8) {$(G_1\otimes\Z/2)\times G_2$};
		\node[below of=4, yshift=0.06cm] (9) {$G_2$};
		\node[below of=5, yshift=0.06cm] (10) {$0,$};
		\draw[->,swap](1) to node {} (2);
		\draw[->](2) to node {} (3);
		\draw[->,swap](3) to node {} (4);
		\draw[->](4) to node {} (5);
		\draw[->](6) to node{} (7);
		\draw[->](7) to node {} (8);
		\draw[->](8) to node {} (9);
		\draw[->](9) to node {} (10);
		\draw[->] (2) to node {$f_1\otimes\Z/2$} (7);
		\draw[->] (3) to node {$(f_1\otimes\Z/2)\times f_2$} (8);
		\draw[->] (4) to node {$f_2$} (9);
	\end{tikzpicture}
\end{center}
so $(\id, f_2, f_1)\in\End\big(\Gamma(X)\big)$. Thus, $\End\big(\Gamma(X)\big)\cong\End(G_1)\times\End(G_2)$.
\end{proof}

However, Proposition \ref{proposition:end2} does not provide a positive answer to Question \ref{question}. Indeed, Remark \ref{remark:negativeRealisability} can easily be adapted to show that $\F_p^3$ is not isomorphic to the direct sum of two endomorphism rings of abelian groups. 

So far we have seen that the product of two endomorphism rings can be realised by choosing a $\Gamma$-sequence in which the homology groups are in a sense isolated in the sequence. We now show how we can take advantage of the interaction between the groups involved in a $\Gamma$-sequence to realise a pullback of endomorphism rings. Recall that given rings $R_1,R_2$ and  $Z$ and morphisms $f_1\colon R_1\to Z$ and $f_2\colon R_2\to Z$, the pullback of $f_1$ and $f_2$ is the ring $R_1\times_Z R_2 = \{(r_1,r_2)\in R_1\times R_2\mid f_1(r_1)=f_2(r_2)\}$.

\begin{proposition}\label{proposition:pullback}
	Let $X$ be an $A_n^2$-polyhedron such that $H_{n+1}(X)$ is the trivial group. Let $R_1 = \{f\in\End\big(H_{n+2}(X)\big)\mid f(\ker b_{n+2})\le \ker b_{n+2}\}$ and $R_2 = \{g\in\End\big(H_n(X)\big)\mid (g\otimes\Z/2)(\Im b_{n+2})\le \Im b_{n+2}\}$. Let $f_1\colon R_1\to\End\big(\Im(b_{n+2})\big)$ be the morphism induced by the first isomorphism theorem and let $f_2\colon R_2\to\End(\Im b_{n+2})$, $g\mapsto (g\otimes\Z/2)|_{\Im b_4}$. Then, $\End\big(\Gamma(X)\big) = R_1\times_{\End(\Im b_{n+2})} R_2$.
\end{proposition}

\begin{proof}
	It is clear that if $(f,0,g)\in \End\big(\Gamma(X)\big)$, then $f$ and $g\otimes\Z/2$ are respectively invariant on the kernel and image of $b_{n+2}$, for otherwise the diagram induced on the $\Gamma$-sequence could not be commutative. Then, the equality $b_{n+2}\circ f = (g\otimes\Z/2)\circ b_{n+2}$ can be reformulated as $f_1(f)=f_2(g)$. The result follows.
\end{proof}

The result above provides us with a different approach to realising the product of two endomorphism rings of abelian groups as $\End\big(\Gamma(X)\big)$ for $X$ an $A_n^2$-polyhedron, although this approach is more restrictive than Proposition \ref{proposition:end2}.

\begin{example}\label{example:anotherEnd2}
	Let $G_1$ and $G_2$ be abelian groups and assume that $G_1$ is free. Let $X$ be the $A_n^2$-polyhedron whose $\Gamma$-sequence is
	\[G_1\xrightarrow{\ 0\ } G_2\otimes\Z/2\xrightarrow{\ \id\ } G_2\otimes\Z/2\longrightarrow 0 \longrightarrow 0.\]
	In this case, $\ker b_{n+2} = G_1$ and $\Im b_{n+2} = \{0\}$, thus $R_1 = \End(G_1)$, $R_2 = \End(G_2)$ and $f_1$ and $f_2$ are both trivial maps. Consequently, $\End\big(\Gamma(X)\big) \cong\End(G_1)\times\End(G_2)$.
\end{example}

Nonetheless, Example \ref{example:anotherEnd2} gives us a major hint on how to prove Theorem \ref{theorem:end3}.

\begin{theorem}[Theorem \ref{theorem:end3}]
	Let $R$ be a ring such that $R\cong\End(G_1)\times\End(G_2)\times\End(G_3)$, where $G_1$ is a free abelian group and $G_2$ and $G_3$ are abelian groups. Then, there exists an $A_n^2$-polyhedron $X$ such that $\End\big(\Gamma(X)\big)\cong R$.
\end{theorem}

\begin{proof}
	Let $X$ be the $A_n^2$-polyhedron whose $\Gamma$-sequence, \eqref{eq:gammaSeqX}, is
	\[G_1\xrightarrow{\ 0\ } G_2\otimes\Z/2\longrightarrow (G_2\otimes\Z/2)\times G_3\longrightarrow G_3 \longrightarrow 0,\]
	and where the the sequence 
	\[0\longrightarrow \operatorname{coker} b_{n+2} = G_2\otimes\Z/2\longrightarrow (G_2\otimes\Z/2)\times G_3\longrightarrow G_3 \longrightarrow 0\]
	is split. Following Proposition \ref{proposition:end2} and Example \ref{example:anotherEnd2} it is immediate to check that for any $f_i\in\End(G_i)$, $i=1,2,3$, $(f_1,f_3,f_2)\in\End\big(\Gamma(X)\big)$. As these are the only possible $\Gamma$-morphisms, $\End\big(\Gamma(X)\big)\cong\End(G_1)\times\End(G_2)\times\End(G_3)$.
\end{proof}

Note, however, that this result is not enough to provide a positive answer to Question \ref{question}, as the endomorphism ring of a non-trivial torsion-free group is always infinite. Thus, if we want to realise $\F_p^3$ using Theorem \ref{theorem:end3} we need to assume that $G_1 = 0$, so we fall in the situation of Proposition \ref{proposition:end2}.

\section{A negative answer to the realisability problem for \texorpdfstring{$\End\big(\Gamma(X)\big)$}{End(Gamma(X))}}

In this section we show that the framework of $A_n^2$-polyhedra is not suitable to produce a complete positive answer to Question \ref{question}. We do so by showing that $\F_p^3$ is not realisable as $\End\big(\Gamma(X)\big)$ for $X$ a finite type $A_n^2$-polyhedron and for $p$ any prime, see Lemmas \ref{lemma:not2} and \ref{lemma:2}.

As we are interested in the realisability of certain finite rings, it would be interesting to characterise the $A_n^2$-polyhedra whose endomorphism rings are finite. Since we are in the setting of polyhedra of finite type, we know from \cite[Corollary 3.2]{Joh72} that $\End(X)$ and $\End\big(H_*(X)\big)$ are isomorphic module torsion. As a consequence, homologically trivial maps must be torsion elements in $\End(X)$. Thus, $\End(X)$ and $\End\big(\Gamma(X)\big)$ are also isomorphic module torsion, from which we immediately deduce the following result.

\begin{corollary}\label{corollary:equivinfinite}
	Let $X$ be an $A_n^2$-polyhedron of finite type, $n\ge 4$. Then $\End(X)$ is finite if and only if $\End\big(\Gamma(X)\big)$ is so.
\end{corollary}

We now prove some technical lemmas we will need in the sequel.  Let us recall the following concept from \cite{CosMenVir20}. 

\begin{definition}(\cite[Definition 4.3]{CosMenVir20})
	Let $f\colon H\to K$ be a morphism of abelian groups. We say that a non-trivial subgroup $A\le K$ is $f$-split if there are groups $B\le H$ and $C\le K$ such that $H\cong A\oplus B$, $K=A\oplus C$ and there exists some $g\colon B\to C$ such that $f = \id_A\oplus g\colon A\oplus B\to A\oplus C$.
\end{definition}

We then have the following result.

\begin{lemma}\label{lemma:splitend}
	Let $X$ be an $A_n^2$-polyhedron of finite type, $n\ge 4$, and let $A\le H_{n+1}(X)$ be an $h_{n+1}$-split subgroup of $H_{n+1}(X)$, thus $H_{n+1}(X)=A\oplus C$ for some abelian group $C$. Then $\End(A)\cong\{(0,f_A\oplus 0_C,0)\mid f_A\in\End(A)\}\le \End\big(\Gamma(X)\big)$.
\end{lemma}

\begin{proof}
	By hypothesis, $H_{n+1}(X)=A\oplus C$, $\pi_{n+1}(X)\cong A\oplus B$ for some abelian group $B$, and  $h_{n+1} $ can be written as  $\id_A\oplus g$ for some morphism $g\colon B\to C$. Namely, the $\Gamma$-sequence of $X$ can be written as
	\[H_{n+2}(X) \xrightarrow{b_{n+2}} \big(H_n(X)\big)\otimes\Z/2 \xrightarrow{i_n} A\oplus B \xrightarrow{\id_A\oplus g} A\oplus C\to 0. \]
	Since $\id_A$ is an isomorphism, the exactness of the $\Gamma$-sequence implies that $\Im i_n\le B$. Thus, for every $f_A\in \End(A)$ we have a commutative diagram
	\begin{center}
		\begin{tikzpicture}
			\tikzset{node distance=0.15\textwidth, auto}
			\node(1) {$H_{n+2}(X)$};
			\node[right of=1, xshift=0.9cm] (2) {$\big(H_n(X)\big)\otimes\Z/2$};
			\node[right of=2, xshift=0.5cm] (3) {$A \oplus B$};
			\node[right of=3, xshift=0.8cm] (4) {$A \oplus C$};
			\node[right of=4] (5) {$0$};
			\node[below of=1, yshift=0.06cm] (6) {$H_{n+2}(X)$};
			\node[below of=2, yshift=0.06cm] (7) {$\big(H_n(X)\big)\otimes\Z/2$};
			\node[below of=3, yshift=0.06cm] (8) {$A\oplus B$};
			\node[below of=4, yshift=0.06cm] (9) {$A \oplus C$};
			\node[below of=5, yshift=0.06cm] (10) {$0.$};
			\draw[->](1) to node {$b_{n+2}$} (2);
			\draw[->](2) to node {$i_n$} (3);
			\draw[->](3) to node {$h_{n+1}$} (4);
			\draw[->](4) to node {} (5);
			\draw[->](6) to node{$b_{n+2}$} (7);
			\draw[->](7) to node {$i_n$} (8);
			\draw[->](8) to node {$h_{n+1}$} (9);
			\draw[->](9) to node {} (10);
			\draw[->](1) to node {$0$} (6);
			\draw[->] (2) to node {$0$} (7);
			\draw[->] (3) to node {$f_A\oplus 0$} (8);
			\draw[->] (4) to node {$f_A\oplus 0$} (9);
		\end{tikzpicture}
	\end{center}
	Consequently, $(0, f_A\oplus 0,0)\in\End\big(\Gamma(X)\big)$.
\end{proof}

The following lemmas are also necessary. 

\begin{lemma}\label{lemma:Hnpgp}
	Let $X$ be an $A_n^2$-polyhedron of finite type, $n\ge 4$, and suppose that $H_n(X) = A\oplus C$ where $A$ is a group of odd order. Then $\End(A)\cong \{(0,0,f\oplus 0)\mid f\in \End(A)\}\le\End\big(\Gamma(X)\big)$.	
\end{lemma}

\begin{proof}
	As $A$ is a group of odd order, $A\otimes \Z/2$ is the trivial group, thus $H_n(X)\otimes\Z/2 = C\otimes \Z/2$. Then, for any $f_A\in\End(A)$, $(f_A\oplus 0_C)\otimes \Z/2$ is the trivial map, giving raise to a commutative diagram 
	\begin{center}
		\begin{tikzpicture}
		\tikzset{node distance=0.15\textwidth, auto}
		\node(1) {$H_{n+2}(X)$};
		\node[right of=1, xshift=0.9cm] (2) {$C\otimes\Z/2$};
		\node[right of=2, xshift=0.5cm] (3) {$\pi_{n+1}(X)$};
		\node[right of=3, xshift=0.8cm] (4) {$H_{n+1}(X)$};
		\node[right of=4] (5) {$0$};
		\node[below of=1, yshift=0.06cm] (6) {$H_{n+2}(X)$};
		\node[below of=2, yshift=0.06cm] (7) {$C\otimes\Z/2$};
		\node[below of=3, yshift=0.06cm] (8) {$\pi_{n+1}(X)$};
		\node[below of=4, yshift=0.06cm] (9) {$H_{n+1}(X)$};
		\node[below of=5, yshift=0.06cm] (10) {$0.$};
		\draw[->](1) to node {$b_{n+2}$} (2);
		\draw[->](2) to node {} (3);
		\draw[->](3) to node {$h_{n+1}$} (4);
		\draw[->](4) to node {} (5);
		\draw[->](6) to node{$b_{n+2}$} (7);
		\draw[->](7) to node {} (8);
		\draw[->](8) to node {$h_{n+1}$} (9);
		\draw[->](9) to node {} (10);
		\draw[->](1) to node {$0$} (6);
		\draw[->] (2) to node {$0$} (7);
		\draw[->] (3) to node {$0$} (8);
		\draw[->] (4) to node {$0$} (9);
		\end{tikzpicture}
	\end{center}
	Consequently, $(0,0,f\oplus 0)\in\End\big(\Gamma(X)\big)$. The result follows.
\end{proof}

\begin{lemma}\label{lemma:Hn+1pgp}
	Let $X$ be an $A_n^2$-polyhedron of finite type, $n\ge 4$, and suppose that $H_{n+1}(X) = G\oplus H$, where $G$ is a $p$-group, $p\ne 2$. Then, $G$ is a $h_{n+1}$-split subgroup of $H_{n+1}(X)$ and $\End(G)\cong\{(0,f\oplus 0,0)\mid f\in \End(G)\}\le\End\big(\Gamma(X)\big)$.
\end{lemma}

\begin{proof}
	Consider the $\Gamma$-sequence of $X$,
	\[H_{n+2}(X) \xrightarrow{b_{n+2}} \big(H_n(X)\big)\otimes\Z/2 \xrightarrow{i_n} \pi_{n+1}(X) \xrightarrow{h_{n+1}} G\oplus H\to 0. \]
	Since $H_n(X)\otimes\Z/2$ is a $2$-group, so is $\Im i_n$. Consequently, by exactness of the $\Gamma$-sequence, $h_{n+1}$ must induce an isomorphism between the $p$-torsion subgroups of $\pi_{n+1}(X)$ and $H_{n+1}(X)$. Thus, $G$ must be a summand of the $p$-torsion subgroup of $\pi_{n+1}(X)$, so there exists $K\le \pi_{n+1}(X)$ such that $\pi_{n+1}(X)\cong G\oplus K$, and we can further assume that $h_{n+1}$ restricts to $G$ as the identity map. Consequently, $G$ is an $h_{n+1}$-split subgroup of $H_{n+1}$, and the result follows from Lemma \ref{lemma:splitend}.
\end{proof}

We now have the necessary tools to prove our main result, Theorem \ref{theorem:main}. We split the proof into Lemma \ref{lemma:not2} and Lemma \ref{lemma:2}, as we consider the cases of even and odd primes separately. We begin with the later. 

\begin{lemma}\label{lemma:not2}
	Let $p$ be an odd prime and $n\ge 4$ be an integer. An $A_n^2$-polyhedron of finite type $X$ such that $\End\big(\Gamma(X)\big)\cong\F_p^3$ does not exist.
\end{lemma}

\begin{proof}
	Assume that one such $A_n^2$-polyhedron exists. Then, since $\End\big(\Gamma(X)\big)$ is finite, by Corollary \ref{corollary:equivinfinite} we can assume that the homology groups of $X$ are finite. As $H_{n+2}(X)$ is a free abelian group, this implies that $H_{n+2}(X)=0$.

	Now by Lemmas \ref{lemma:Hnpgp} and \ref{lemma:Hn+1pgp}, $H_n(X)$ and $H_{n+1}(X)$ must be a direct sum of a $2$-group and a $p$-group. Indeed, if either of the two groups had a summand $G$ which is a $q$-group, $q\not\in\{2,p\}$, then $\End(G)\le\End\big(\Gamma(X)\big)$. But $\End(G)$ contains elements of (additive) order $q$, giving raise to a contradiction.

	Write $H_{n+1}(X) = A\oplus C$ where $A$ is a $p$-group and $C$ is a $2$-group. By Lemma \ref{lemma:Hn+1pgp}, $A$ is $h_{n+1}$-split, so there exists $B\le\pi_{n+1}(X)$ such that $\pi_{n+1}(X)=A\oplus B$ and $h_{n+1}=\id_A\oplus g$ for some group homomorphism $g\colon B\to C$. Similarly, write $H_n(X) = A'\oplus C'$, where $A'$ is a $p$-group and $C'$ is a $2$-group. Then, the $\Gamma$-sequence of $X$ is 
	\[0\longrightarrow C'\otimes\Z/2 \xrightarrow{\ i_n\ } A\oplus B\xrightarrow{\id_A\oplus g} A\oplus C\longrightarrow 0.\]
	Furthermore, as $C'\otimes\Z/2$ is a $2$-group and $A$ is a $p$-group, $p\ne 2$, it follows that $\Im i_n\le B$. 

	Thus, the elements of $\End\big(\Gamma(X)\big)$ are of the form $(0, f_A\oplus f_C, f_{A'}\oplus f_{C'})$, where $f_A\in\End(A)$ and $f_{A'}\in\End(A')$ are any two endomorphisms, and $f_C\in\End(C)$ and $f_{C'}\in\End(C')$ are endomorphisms giving raise to a commutative diagram
	\begin{center}
		\begin{tikzpicture}
			\tikzset{node distance=0.15\textwidth, auto}
			\node(1) {$0$};
			\node[right of=1, xshift=0.9cm] (2) {$C'\otimes\Z/2$};
			\node[right of=2, xshift=0.5cm] (3) {$B$};
			\node[right of=3, xshift=0.8cm] (4) {$C$};
			\node[right of=4] (5) {$0$};
			\node[below of=1, yshift=0.06cm] (6) {$0$};
			\node[below of=2, yshift=0.06cm] (7) {$C'\otimes\Z/2$};
			\node[below of=3, yshift=0.06cm] (8) {$B$};
			\node[below of=4, yshift=0.06cm] (9) {$C$};
			\node[below of=5, yshift=0.06cm] (10) {$0,$};
			\draw[->](1) to node {} (2);
			\draw[->](2) to node {} (3);
			\draw[->](3) to node {$g$} (4);
			\draw[->](4) to node {} (5);
			\draw[->](6) to node{} (7);
			\draw[->](7) to node {} (8);
			\draw[->](8) to node {$g$} (9);
			\draw[->](9) to node {} (10);
			\draw[->] (2) to node {$f_{C'}\otimes\Z/2$} (7);
			\draw[->] (3) to node {$\Omega$} (8);
			\draw[->] (4) to node {$f_C$} (9);
		\end{tikzpicture}
	\end{center}
	for some $\Omega\in\End(B)$. As a consequence, if we let $R$ denote the ring of endomorphisms $(f_C,f_{C'})\in\End(C)\oplus\End(C')$ giving raise to such a commutative diagram, then $\End\big(\Gamma(X)\big)\cong\End(A)\oplus\End(A')\oplus R$.

	Now notice that $R\le \End(C)\oplus\End(C')$, where $C$ and $C'$ are $2$-groups. Thus the group underlying $R$ is itself a $2$-group. As we want $\End\big(\Gamma(X))$ to be a $p$-group, we can assume that $R$ is trivial. Finally, we obtain that $\End\big(\Gamma(X)\big)\cong\End(A)\oplus\End(A')$ where $A$ and $A'$ are $p$-groups. As neither $\F_p^2$ nor $\F_p^3$ are realisable as $\End(A)$ for $A$ any $p$-group, we finally deduce that $\End\big(\Gamma(X)\big)$ cannot be isomorphic to $\F_p^3$.
\end{proof}

We now prove the result for the prime $p=2$. 

\begin{lemma}\label{lemma:2}
	Let $n\ge 4$ and $m\ge 1$ be integers. An $A_n^2$-polyhedron of finite type $X$ such that $\End\big(\Gamma(X)\big)\cong\F_{2^m}^3$ does not exist.
\end{lemma}

\begin{proof}
	Recall from Remark \ref{remark:CMV2applicable} that the group of units of $\End\big(\Gamma(X)\big)$ is isomorphic to $\E(X)/\E_*(X)$. In \cite[Theorem 1.1]{CosMenVir20} we proved that, for $n\ge 3$, an $A_n^2$-polyhedron of finite type such that $\E(X)/\E_*(X)$ is a non-trivial finite group of odd order does not exist. But the group of units of $\F_{2^m}^3$ has order $(2^m-1)^3$, which is an odd number greater than 1 whenever $m>1$. Then $\F_{2^m}^3$ is not realisable, whenever $m>1$.

	Assume now that there is an $A_n^2$-polyhedron of finite type $X$ such that $\End\big(\Gamma(X)\big)\cong\F_2^3$. Then, $\E(X)/\E_*(X)\cong\Aut\big(\Gamma(X)\big)$ is trivial. As a consequence of the development in \cite[Theorem 1.1, Lemma 4.5]{CosMenVir20}, only four $\Gamma$-sequences without non-trivial automorphisms exist. Let us consider each of them separately.

	Assume first that $H_n(X)=\Z/2$ and $H_{n+1}(X)=0$. Then $\End\big(\Gamma(X)\big)\le\End(\Z/2)\cong\Z/2$. The same happens when $H_n(X)=0$ and $H_{n+1}(X) = \Z/2$. Thus these sequences cannot be used to realise $\F_2^3$.

	Consider now $X$ an $A_n^2$-polyhedron where $H_n(X)=H_{n+1}(X)=\Z/2$ and whose $\Gamma$-sequence is
	\[0\to\Z/2\to\Z/2\oplus\Z/2\to\Z/2\to 0.\]
	We know from the proof of Proposition \ref{proposition:end2} that $\End\big(\Gamma(X)\big)\cong \F_2^2$, so this sequence cannot be used either.

	The remaining one is associated to an $A_n^2$-polyhedron $X$ with $H_n(X)=H_{n+1}(X)=\Z/2$, its $\Gamma$-sequence being
	\[0 \to \Z/2\to\Z/4\to\Z/2\to 0.\]
	Therefore, $\End\big(\Gamma(X)\big)\le\End(\Z/2)\oplus\End(\Z/2)\cong(\Z/2)^2$. We conclude that $\F_2^3$ is not realisable.
\end{proof}

Theorem \ref{theorem:main} is then an immediate consequence of Lemma \ref{lemma:not2} and Lemma \ref{lemma:2}.

\begin{remark}
	As a consequence of \cite[Theorem 1.1]{CosMenVir20}, if $X$ is an $A_n^2$-polyhedron of finite type in the stable range, the group of units of $\End\big(\Gamma(X)\big)$ is either the trivial group or it has elements of even order. This comes at no surprise, since it is known that for $X$ a path-connected, finite CW-complex in the stable range, $1_X\not\simeq -1_X$, \cite[Proposition 3]{Kah79}, thus $-1_X$ is an element of order two in the group of units of $\End(X)$. In order for $-1_X$ to induce a trivial map on $\Gamma(X)$, the homology groups of $X$ must be elementary abelian $2$-groups, providing a natural explanation of the results obtained in \cite[\textsection4]{CosMenVir20}. 
\end{remark}


\end{document}